\documentclass[10pt]{amsart}
\usepackage{amsmath, amssymb, amsthm, graphicx, hyperref}

\newtheorem{theorem}{Theorem}[subsection]
\newtheorem{lemma}[theorem]{Lemma}
\newtheorem{corollary}[theorem]{Corollary}
\newtheorem{definition}[theorem]{Definition}
\newtheorem{remark}[theorem]{Remark}
\newtheorem{claim}[theorem]{Claim}

\newtheorem{proposition}[theorem]{Proposition}

\begin{document}
\author{Leonid A. Bunimovich}
\address{School of Mathematics, Georgia Institute of Technology,
        Atlanta, GA 30332}
\email{bunimovh@math.gatech.edu}
\author{Alex Yurchenko}
\address{School of Mathematics, Georgia Institute of Technology,
        Atlanta, GA 30332}
\email{yurchenk@math.gatech.edu}
\title[Where to place a hole to achieve a maximal escape rate.]
{Where to place a hole to achieve a maximal escape rate.}
\thanks{}
\date{\today}
\numberwithin{equation}{section}

\begin{abstract}
A natural question of how the survival probability depends upon a position of a hole was seemingly never
addressed in the theory of open dynamical systems. We found that this dependency could be very essential. The
main results are related to the holes with equal sizes (measure) in the phase space of strongly chaotic maps.
Take in each hole  a periodic point of minimal period. Then the faster escape occurs through the hole where this
minimal period assumes its maximal value. The results are valid for all finite times (starting with the minimal
period) which is unusual in dynamical systems theory where typically statements are asymptotic when time tends to
infinity. It seems obvious that the bigger the hole is the bigger is the escape through that hole. Our results
demonstrate that generally it is not true, and that specific features of the dynamics may play a role comparable
to the size of the hole.
\end{abstract}
\maketitle

\section{Introduction.}
The theory of open dynamical systems is (naturally) much less
developed than of the closed ones. Basically so far the problems
studied were on the existence of conditionally invariant measures,
their properties, and the existence of the escape rates
\cite{Baladi}, \cite{Bruin}, \cite{Chernov1}, \cite{Chernov2},
\cite{Chernov3}, \cite{Collet}, \cite{Dem1}, \cite{Liverani},
\cite{Yorke}, \cite{Chernov4}.

In this paper we address a natural question which, to the best of
our knowledge, has not been studied so far. Obviously, if one
enlarges a hole then the escape rate of the orbits will increase as
well (or, at least, it cannot decrease). Consider, however, two
holes of the same size (measure), placed at the different positions
in the phase space of the dynamical system under study. Would the
escape rates through these holes be equal?

We demonstrate that the answer to this question could be both "yes" and "no".
In case when there exists a group of measure preserving translations of the phase space which commute
with the dynamics the answer is "yes". It is quite natural and intuitive answer which
is justified in Section \ref{sec_rotation} of our paper.
However the dynamics of these systems is quite regular.

Much less trivial is the question of what other factors, besides the size of the hole,
can influence the escape rate.
In particular, what can generate different escape rates through two holes of the same size?

Consider a system with strongly chaotic dynamics. For many classes of such systems it is known
that there exists infinitely many periodic orbits of infinitely many periods and
that the periodic orbits are everywhere dense in the phase space.
Therefore in each hole there are infinitely many periodic points.

Our approach is based on the idea that the faster escape occurs
through a hole whose preimages overlap less than the ones of another
hole. This idea leads to the following procedure (algorithm): 1)
find in each hole a periodic point with minimal period; 2) compare
these periods. We claim that the escape will be faster through a
hole where this minimal period is bigger. This claim is justified
for various classes of dynamical systems with strongly chaotic
behavior and Markov holes in Section \ref{s_3254125}.

We also computed the local escape rate and demonstrated that for all
nonperiodic points this value is the same while at the periodic
points the escape "slows down" and it assumes smaller values at the
periodic points with smaller period.

Thus we demonstrate that the dynamical factors could be as important for the escape as the size of the hole.
In fact it is possible that the escape rate through a larger hole could be less
than the escape rate through a smaller hole.

An important and a new feature of our results is that they hold for
all finite times starting with some moment of time, in comparison to
the usual setup in the theory of dynamical systems where one deals
with the asymptotic properties at infinite time.

For more general classes of dynamical systems not only the
distribution of the periodic points, but other characteristics of
dynamics, e.g. distortion, may also contribute to the process of
escape. This will be considered in a future publication.

The structure of the paper is the following one. Section
\ref{sec:def} contains necessary definitions and some auxiliary
results. Section \ref{sec_rotation} deals with the case where escape
rate does not depend on the position of the hole. Section
\ref{s_3254125} presents the main results of the paper. Section
\ref{sec:more} deals with some generalization and, finally, Section
\ref{sec:end} contains concluding remarks.

\section{Definitions and Some Technical Results.} \label{sec:def}
Consider a discrete time dynamical system given by a measure-preserving map
$$
\hat{T}: \hat{M} \rightarrow \hat{M},
$$
where $\hat{M}$ is a Borel probability space with the measure $\lambda$.
Let $\mathcal B$ be the Borel $\sigma$-algebra on $\hat{M}$
with respect to $\lambda$.

\subsection{Recurrences}

Here we define some notions related to the recurrence properties of the dynamical system.

\begin{definition}
The \emph{Poincar\'{e} recurrence time} of a subset $A \in \mathcal B$ of a positive measure is a positive
integer $\tau(A) \le +\infty$ given by
\begin{equation}
\tau_{\hat{T}} (A) = \inf_{n \ge 1} \{ n :  \lambda(\hat{T}^n(A) \cap A ) > 0 \}.
\end{equation}
\end{definition}
If there is no ambiguity about which map we are considering, then we drop the subscript and use $\tau(A)$
instead. If the Poincar\'{e} recurrence time is finite, then it is the smallest integer $n$ such that the $n$th
iterate of $A$ under $\hat{T}$ intersects $A$ nontrivially (in this case nontrivially means that intersection has
a non-zero measure).

According to the Poincar\'{e} Recurrence Theorem (see, for example Theorem 1.4 in \cite{Walters}) for the spaces
of finite measure the Poincar\'{e} recurrence time of any measurable set of positive measure is finite.

Next, we list a few properties of Poincar\'{e} recurrence time which will be used later. These statements follow
easily from the definition.

\begin{proposition} Let $A$ and $B$ be two measurable sets. Then
\begin{enumerate}
    \item if $A \subset B$ then $\tau(A) \ge \tau(B)$;
    \item $\tau(A) = \tau(\hat{T}^{-1}(A))$,
\end{enumerate}
where $\hat{T}^{-1}(A)$ is a complete  preimage of $A$.
\end{proposition}

For $n \ge 0$ and $A \in \mathcal B$, define the following (measurable) sets \cite{Dem1},
\begin{align*}
\Omega_n (A) = \left\{ x \in \hat{M} : \exists j \in \mathbb N, 0 \le j \le n, \hat{T}^j x \in A \right\} =  \cup_{i=0}^{n} \hat{T}^{-i}(A), \\
\Theta_n (A) = \left\{ x \in \hat{M} : \hat{T}^n x \in A, \hat{T}^j x \notin A, j = 0, \ldots, n-1 \right\},
\end{align*}
where $\hat{T}^{-i}(A)$ is a complete preimage of $A$ under $\hat{T}^i$. The set $\Omega_n(A)$ consists of all
points which orbits enter $A$ after no more then $n$ iterates. The set $\Theta_n(A)$ consists of all points which
orbits enter $A$ at first time exactly on $n$th iterate. Note that $\Omega_0(A) = A$ and $A \subset \Omega_n(A)$,
$\forall n \in \mathbb N$. It is easy to see that these sets have the following properties.

\begin{proposition}Let $A$ be a measurable set. Then
\begin{enumerate}
    \item $\Omega_i(A)$ is a nondecreasing sequence of sets;
    \item $\Theta_i(A) \cap \Theta_j(A)  = \emptyset$ if $i \ne j$, $i,j > 0$;
    \item $\Omega_n(A) = \cup_{i=0}^n \Theta_i(A)$.
\end{enumerate}
\end{proposition}

\subsection{Open dynamical systems.}\label{sec:holes}

Let $A$ be a measurable set and let $M = \hat{M} \backslash A$.
We define an open dynamical system (system with a "hole" $A$) as a map
$$
T: M \rightarrow \hat{M},
$$
where $T : = \hat{T}_{|M}$ is a restriction of $\hat{T}$ to $M$. We
keep track of the orbits while they stay outside the "hole" $A$, and
after they enter a hole we no longer consider these orbits (they
just "disappear"). So we can talk about iterates of $T$ instead of
$\hat{T}$ as long as orbit stays outside $A$. Alternatively, one can
redefine $\hat{T}$ in such a way that it is an identity map on $A$.
We will use the former approach. Note that we use a hat over a
letter to denote an object in a closed system and letters without a
hat for corresponding objects in the open system.

\subsection{Escape rate.}

\begin{definition}
The (exponential) escape rate into the hole $A$
is a nonnegative number $\rho(A)$ given by
\begin{equation}
\rho(A) = - \lim_{n \rightarrow \infty} \frac{1}{n} \ln  \lambda \left( \hat{M} \backslash \Omega_n (A) \right) ,
\end{equation}
if this limit exists.
\end{definition}

The number $\lambda \left( \hat{M} \backslash \Omega_n (A) \right) = 1 - \lambda \left( \Omega_n (A) \right)$
(sometimes called a \emph{survival probability}) is the measure of the set that does not escape into the hole in
$n$ iterations. Hence, the escape rate represents the average rate at which orbits enter the hole. The larger the
escape rate is, the faster the "mass" escapes from the system into the hole $A$.

We will only consider  systems in which almost every orbit
eventually enters the hole, i.e. systems which satisfy the following condition
\begin{equation}\label{H1}
\sum_{i=0}^{\infty} \lambda (\Theta_i (A)) = 1. \tag{H1}
\end{equation}
Any ergodic system would be an example of such a system. In that
case property \ref{H1} holds for any measurable hole of positive
measure. On the other hand if we consider a system with a globally
attracting set $A$ then the property \ref{H1} holds only for that
set $A$ and any set which contains $A$.

The next proposition lists a few simple but useful properties of the escape rate.

\begin{proposition} \label{prop_23423}
Let $A$ and $B$ be two measurable sets. Assume that $\rho(A)$ and
$\rho(B)$ exist. Then,
\begin{enumerate}
    \item if $A \subset B$ then $\rho(A) \le \rho(B)$;
    \item $\rho(A) = \rho(\hat{T}^{-1}(A))$;
    \item for any finite positive integer $m$, $\rho(A) = \rho \left( \cup_{i=0}^m \hat{T}^{-i}A \right)$;
    \item if $B \subset \hat{T}^{-k}A$ for some $k >0$ then $\rho(A \cup B) = \rho(A)$;
    \item if there exists $m$ such that $\hat{M} = \cup_{i=0}^m \hat{T}^{-i}A$
        then $\rho(A) = + \infty$.
\end{enumerate}
\end{proposition}

The first part of the proposition says that the size of the hole is one of the
factors that determines the escape rate.
As we will see later, it is not necessarily the only one or even the dominant one.
Moreover, c) and d) state that, in principle, we can have a system in which holes of different size
have the same escape rate.

Instead of looking at the measure of the set that does not
enter a hole during the first $n$ iterations, sometimes it is more convenient to consider
the set which enters the hole for the first time on exactly $n$th iteration (but not earlier).
The following lemma illustrates how we can accomplish that.

\begin{lemma} \label{l_eq}
Suppose that condition \ref{H1} holds and the escape rate,
$\rho(A)$, exists. Then
$$
\rho(A) = - \lim_{n \rightarrow \infty} \frac{1}{n} \ln \lambda (\Theta_n(A)).
$$
\end{lemma}
\begin{proof}
Let $a_n = \lambda (\Theta_n(A))$ and assume that $- \lim_{n
\rightarrow \infty} \frac{1}{n} \ln a_n = \alpha$. Then $\forall
\epsilon \in (0, \alpha)$ $\exists N \in \mathbb N$ such that
$\forall n \ge N$ one has that
$$
- \epsilon - \alpha \le \frac{1}{n} \ln a_n \le \epsilon - \alpha
$$
or, equivalently,
$$
e^{-n(\alpha + \epsilon)} \le a_n \le e^{-n(\alpha - \epsilon)}.
$$

Next, observe that if $\rho(A)$ exists, then it is given by
\begin{align*}
\rho(A) = - \lim_{n \rightarrow \infty} \frac{1}{n} \ln \left( 1 - \sum_{i=0}^{n} \lambda (\Theta_i(A)) \right) \\
= - \lim_{n \rightarrow \infty} \frac{1}{n} \ln \left( \sum_{i=n+1}^{\infty} \lambda (\Theta_i(A)) \right)
= - \lim_{n \rightarrow \infty} \frac{1}{n} \ln \left( \sum_{i=n+1}^{\infty} a_i \right).
\end{align*}
For $n \ge N$ we have
$$
\sum_{i=n+1}^{\infty} e^{-i(\alpha + \epsilon)} \le \sum_{i=n+1}^{\infty} a_i \le \sum_{i=n+1}^{\infty} e^{-i(\alpha - \epsilon)}
$$
or, equivalently,
$$
\frac{e^{-(n+1)(\alpha + \epsilon)}}{1- e^{-(\alpha + \epsilon)}} \le \sum_{i=n+1}^{\infty} a_i \le \frac{e^{-(n+1)(\alpha - \epsilon)}}{1- e^{-(\alpha - \epsilon)}}.
$$
Taking the logarithm of both sides, dividing by $n$, and letting $n$
tend to infinity we comlete the proof.
\end{proof}

Recall now the notion of metric conjugacy which will play an important role in what follows. Note that for
Lebesgue probability spaces metric conjugacy is equivalent to two maps being isomorphic (see, for example,
Theorem 2.5 and 2.6 in \cite{Walters}). We use the following definition.

\begin{definition}
Let $T_i$ be a measure-preserving transformation of the Lebesgue probability space $(X_i, \mathcal B_i,
\lambda_i)$, $i=1,2$, where $\mathcal B_i$ is a Borel $\sigma$-algebra on $X_i$ and $\lambda_i$ is a probability
measure. We say that $T_1$ and $T_2$ are \emph{metrically conjugate} if there exist $M_i \in \mathcal B_i$ with
$\lambda_i (M_i) = 1$ and $ T_i(M_i) \subset M_i$ and there is a invertible measure-preserving transformation
(called metric conjugacy) $F : M_2 \rightarrow M_1$ such that
$$
F \circ T_2 (x) = T_1 \circ F (x), \quad \forall x \in M_2.
$$
\end{definition}

The following result states that if two systems are metrically conjugate,
then the escape rates into the corresponding holes and
Poincar\'{e} return times of these holes are the same for both systems.

\begin{lemma}\label{l_32521}
Let $T_1$ and $T_2$ be two metrically conjugate measure-preserving
transformations on the Borel probability spaces $(X_1, \mathcal B_1,
\lambda_1)$ and $(X_2, \mathcal B_2, \lambda_2)$, correspondingly,
with a conjugacy map $F:  (\mathcal B_2, \lambda_2) \rightarrow
(\mathcal B_1, \lambda_1)$. Suppose also that $T_2$ satisfy
condition \ref{H1}. Then $\forall A \in \mathcal B_2$ we have
\begin{enumerate}
    \item $\rho_{T_2} (A) = \rho_{T_1} (F(A))$,
    \item $\tau_{T_2}(A) = \tau_{T_1}(F(A))$.
\end{enumerate}

\end{lemma}
\begin{proof}

\begin{enumerate}
    \item
Let $A \in \mathcal B_2$ and, as above, define two sets
\begin{align*}
\Theta_n^2 (A) = \left\{ x \in X_2 : T_2^n x \in A, T_2^j x \notin A, j = 0, \ldots, n-1 \right\}; \\
\Theta_n^1 (F(A)) = \left\{ y \in X_1 : T_1^n y \in F(A), T_1^j y \notin F(A), j = 0, \ldots, n-1 \right\}.
\end{align*}
Then the escape rates for two systems are given by
\begin{align} \label{e_41235}
\rho_{T_1} (A) = - \lim_{n \rightarrow \infty} \frac{1}{n} \ln \lambda_1 (\Theta_n^1 (F(A))); \\
\rho_{T_2} (A) = - \lim_{n \rightarrow \infty} \frac{1}{n} \ln \lambda_2 (\Theta_n^2(A)). \nonumber
\end{align}

\begin{claim} $\forall n\ge 1$,
$$
F \left( \Theta_n^2 (A) \right) = \Theta_n^1 \left( F(A) \right).
$$
\end{claim}
\begin{proof}
\begin{align*}
\Theta_n^1 \left( F(A) \right) = \left\{ y \in X_1 : T_1^n y \in F(A), T_1^j y \notin F(A), j = 0, \ldots, n-1 \right\} \\
= \left\{ y \in X_1 : F^{-1} T_1^n y \in A, F^{-1} T_1^j y \notin A, j = 0, \ldots, n-1 \right\} \\
= \left\{ y \in X_1 :  T_2^n F^{-1} y \in A, T_2^j F^{-1} y \notin A, j = 0, \ldots, n-1 \right\} \\
= \left\{ F(x) \in X_1 :  T_2^n x \in A, T_2^j x \notin A, j = 0, \ldots, n-1 \right\} \\
= F \left( \left\{ x \in X_2 :  T_2^n x \in A, T_2^j x \notin A, j =
0, \ldots, n-1 \right\} \right)  = F \left( \Theta_n^2 (A) \right).
\end{align*}
\end{proof}

By the previous claim and definition of the conjugacy map we have
$$
\lambda_2 \left( \Theta_n^2 (A) \right) = \lambda_1 \left( F \left( \Theta_n^2 (A) \right) \right) = \lambda_1 \left( \Theta_n^1 (F(A)) \right).
$$
Hence one has that $\rho_{T_2} (A) = \rho_{T_1} (F(A))$, $\forall A \in \mathcal B_2$.

    \item Suppose that $\tau_{T_2}(A) = N < \infty$. That means that
$$
\lambda \left( T ^N _2(A) \cap A \right) > 0
$$
and
$$
\lambda \left( T ^k _2(A) \cap A \right) = 0, \quad k < N.
$$
But
$$
\lambda \left(T_1^N (F(A)) \cap F(A) \right) = \lambda \left(F (T_2^N (A)) \cap F(A) \right) >0
$$
and similarly
$$
\lambda \left(T_1^k (F(A)) \cap F(A) \right)=0, \quad k <N.
$$
Hence, $\tau_{T_1}(F(A)) = N$.
\end{enumerate}
\end{proof}

\section{Escape rate for the ergodic group rotations.}\label{sec_rotation}

Suppose that the phase space $\hat{M}$ is a compact connected metric group. Let  $S_a: \hat{M} \rightarrow
\hat{M}$ ge a group rotation defined as
$$
S_a(x) = ax,
$$
for some $a \in \hat{M}$. Then there is the unique any rotation invariant probability measure, $\lambda$, called
Haar measure (see, e.g. \cite{Katok}).

The following simple statement claims that if a group rotation $S$ commutes with the dynamics, i.e.
\begin{equation} \label{eq_342}
\tag{H2} \hat{T}^{-1} S = S \hat{T}^{-1},
\end{equation}
then the escape rate is invariant when we rotate the hole by $S$.

\begin{lemma} \label{thm_3421}
Suppose that $\hat{T}^{-1} S (A) = S \hat{T}^{-1} (A)$ for some $S \in \mathbb G$ and $A \in \mathbb B$. Also,
assume that $\rho_{\hat{T}}(A)$ exists. Then $\rho_{\hat{T}} (A) = \rho_{\hat{T}} \circ S (A)$.
\end{lemma}
\begin{proof}
Let $A \in \mathbb B$ as in the statement of the theorem. Then,
\begin{align*}
\Theta_n (S(A)) = \left\{ x \in \hat{M} : \hat{T}^n x \in S(A), \hat{T}^j x \notin S(A), j = 0, \ldots, n-1 \right\} \\
= \left\{ x \in \hat{M} : \hat{T}^n S^{-1}x \in A, \hat{T}^j S^{-1}x \notin A, j = 0, \ldots, n-1 \right\} \\
= \left\{ Sy \in \hat{M} : \hat{T}^n y \in A, \hat{T}^j y \notin A,
j = 0, \ldots, n-1 \right\} = S (\Theta_n (A)).
\end{align*}
Therefore, by Lemma \ref{l_eq} the result follows.
\end{proof}

Assume now that $\hat{T} = \hat{T}_a$ is an ergodic rotation of $\hat{M}$ given by
$$
\hat{T}_a(x) = ax,
$$
for some $a \in \hat{M}$. For the rotations we use the following property as a definition of ergodicity.

\begin{proposition}[see, for example, Theorem 1.9 in \cite{Walters}]
Let $\hat{M}$ be a compact group and $\hat{T}_a$ a rotation of $\hat{M}$. Then $\hat{T}_a$ is ergodic iff
$\{a^n\}_{n= - \infty}^{+ \infty}$ is dense in $\hat{M}$.
\end{proposition}

The simplest example of this class of dynamical systems is the irrational rotations of the circle. It is known
(see, for example, Theorem 1.9 in \cite{Walters}) that if there is an ergodic rotation of $\hat{M}$ then
$\hat{M}$ must be Abelian. In that case conditions \ref{H1} and \ref{eq_342} are satisfied so Lemma
\ref{thm_3421} is applicable. Hence, the escape rate for any hole is independent of the position of that hole.
Moreover, one can compute the escape rate for any hole that contains an open set. It turns out that it is
infinite because in the case of ergodic rotation all orbits escape within finite amount of time.

\begin{theorem}
For any ergodic rotation $\hat{T}_a =ax$ of the compact connected metric group $\hat{M}$
and any hole $A$ that contains an open ball the escape rate is infinite.
\end{theorem}
\begin{proof}
Let $dist(\cdot,\cdot)$ be a metric on $\hat{M}$. Suppose that
$V(x,\varepsilon)$ is an open ball of radius $\varepsilon$ centered
at $x$ that is contained in the hole $A$. Since $\hat{M}$ is a
compact metric space we can find a finite covering by open balls of
radius $\frac{1}{4}\varepsilon$. Let $\left\{ V_i(x_i,\frac{1}{4}
\varepsilon) \right\}_{i=1}^m$ be that covering.

It follows from ergodicity that if $\hat{T}_a$ is ergodic
then the set $\{a^i \}_{i=1}^{+\infty} $ is dense in $\hat{M}$.
Therefore we can find $n$ and $\{n_j \} _{j=1}^m$ such that
\begin{align*}
    dist(a^n,x) < \frac{1}{4} \varepsilon; \\
    dist(a^{n_j},x_j) < \frac{1}{4} \varepsilon, \quad \forall j=1 \ldots m;\\
    n > max_{j=1 \ldots m} \{n_j \}.
\end{align*}

Then we have that
\begin{align*}
dist(\hat{T}_a^{n-n_j}x_j,x) \le dist(\hat{T}_a^{n-n_j}x_j, a^n) + dist(a^n,x) \\
= dist(\hat{T}_a^{n-n_j}x_j, \hat{T}_a^{n-n_j} a^{n_j}) + dist(a^n,x) \\
= dist(x_j, a^{n_j}) + dist(a^n,x) < \frac{1}{2} \varepsilon.
\end{align*}
Therefore, $\forall  j=1 \ldots m$, $\hat{T}_a^{n-n_j} V_j \subset V$.
Thus every ball $V_j$ will be mapped into the hole in finite number of steps.
Since every set can be covered by these balls we get that any set is mapped
into the hole in finite number of iterations. This finishes the proof.

\end{proof}

\section{Escape rate for the expanding maps of the interval.} \label{s_3254125}

In this section we look at the examples of the dynamical systems in which the position of the hole plays an
important role in determining the escape rate. We consider some classes of the uniformly expanding maps of the
interval that have a finite Markov partition. These systems are the examples of so called chaotic dynamical
systems.

Consider first $\hat{T}: [0,1] \rightarrow [0,1]$ given by
$$
\hat{T}x = \kappa x \mod 1,
$$
where $\kappa$ is an integer larger then one. This map preserves the
Lebesgue measure on $[0,1]$. Without any loss of generality one can
assume that $\kappa = 2$.

Fix $N \in \mathbb N$ and let $\mathcal I_N = \left\{ I_{i,N} \right\}_{i=1}^{2^N}$ be
the partition consisting of the pre-images of the elements of Markov partition
$\{ [0, 0.5], [0.5, 1] \}$ of $[0,1]$ given by
$$
I_{i,N} = \left[ \frac{i-1}{2^N}, \frac{i}{2^N} \right], \quad i=1 \ldots 2^N.
$$
Define a partition $\mathcal I^k_N = \left\{ I_{j,N+k} \right\}_{j=1}^{2^{N+k}}$
as $k^{th}$ preimage of the partition $\mathcal I_N$, i.e.
for each $j$, $j=1 \ldots 2^{N+k}$, there is $i$, $i=1 \ldots 2^N$, such that
$\hat{T}^k I_{j,N+k} = I_{i,N}$.
Observe that $\mathcal I^k_N$ are Markov partitions themselves.

Consider now an open dynamical system defined by the map $\hat{T}$ and the hole $I_{i,N}$,
$$
T_{i,N} : [0,1] \backslash I_{i,N} \rightarrow [0,1]
$$
as in section \ref{sec:holes}. For each $i$,  $i=1 \ldots 2^N$, we
have different open dynamical system with a corresponding hole
$I_{i,N}$ (we refer to this hole as to a \emph{Markov hole} because
$I_{i,N}$ is an element of Markov partition). Define the
Poincar\'{e} recurrence time of the hole, $\tau(I_{i,N})$, escape
rate into the hole, $\rho(I_{i,N})$, and the set
\begin{equation*}
\Omega_n (I_{i,N}) = \left\{ x \in [0,1] : \exists j, 0 \le j \le n, \quad \hat{T}^j x \in I_{i,N} \right\}, \quad n \ge 0.
\end{equation*}

In Section \ref{s_51223415} we will show that escape rate is well
defined and then, in Section \ref{s_51223416}, that it depends not
only on the size of the hole but also on its position. More
precisely, we prove the following.

\textbf{Theorem (Main Theorem).}
\emph{Let $I_{i,N}$ and $I_{j,N}$ be two Markov holes for the doubling map.
Suppose that $\tau(I_{j,N}) > \tau(I_{i,N})$. Then,
$$
\rho(I_{j,N}) > \rho(I_{i,N}).
$$
Moreover, for all $n \ge  \tau(I_{i,N})$,
$$
1 - \lambda \left( \Omega_n (I_{j,N})\right) < 1 - \lambda \left( \Omega_n (I_{i,N})\right).
$$
}

In the section \ref{sec_4534345234} we show that asymptotically as we decrease the size of the hole
escape rate is proportional to the size of the hole.

\textbf{Theorem (Local Escape).}
\emph{ Let $x \in [0,1]$ and $\{A_n (x) \}_{n=1} ^{\infty}$ is a sequence of
nested decreasing intervals with $x = \cap _{n=1} ^{\infty} A_n (x)$ for all $n$. The following statements hold:
\begin{enumerate}
    \item if $x$ is a periodic point of period $m$ then
$$
\lim_{n \to \infty} \frac{\rho(A_{n}(x))}{\lambda(A_{n}(x))} = 1 - \frac{1}{2^m };
$$
    \item if $x$ is a non-periodic point and $x \neq s2^{-k}, s,k \in \mathbb Z^{+}$, then
$$
\lim_{n \to \infty} \frac{\rho(A_{n}(x))}{\lambda(A_{n}(x))} = 1.
$$
\end{enumerate}}

Moreover, for a sequence of shrinking Markov holes (Section \ref{sec_4534626})
the corresponding sequence of escape rates is a monotone one.

\textbf{Theorem (Monotonicity).}
\emph{
$$
\max_{1 \le i \le 2^{N+1}} \rho (I_{i,N+1}) = \min_{1 \le j \le 2^N}
\rho (I_{j,N}).
$$
}

First, we state some known results about doubling map
and then reformulate the problem in terms of the symbolic dynamics.

\subsection{Preliminary results.}\label{s_3254125a}

Let $p_{i,N}(n)$ be a number of points of period (not necessary the minimal one) $n$
in the hole $I_{i,N}$. Due to the fact that $\mathcal I_N$ is a Markov partition,
for each j, $j = 1 \ldots 2^{N+k}$,
$\hat{T}^k_{i,N}(I_{j,N+k}) = I_{s,N}$ for some $s$, $s = 1 \ldots 2^{N}$.
Let $f^k_{i,N}$ be a number of elements of partition $\mathcal I^k_N$ that do not
enter the hole $I_{i,N}$ in the first $k$ iterations, i.e.
$$
f^k_{i,N} =  \# \left\{ j : I_{j,N+k} \in \mathcal I^k_N; \lambda \left( \hat{T}^s(I_{j,N+k}) \cap I_{i,N} \right) = 0; s=0, \ldots , k \right\}.
$$

The next two results show how to compute the Poincar\'{e} recurrence time and
the escape rate for any Markov hole.
\begin{proposition}
The Poincar\'{e} recurrence time of the hole $I_{i,N}$ is equal to
the period of a periodic point contained in $I_{i,N}$ having the smallest period, i.e.
$$
\tau(I_{i,N}) = \min_{n \ge 1} \{n : p_{i,N}(n) > 0 \}.
$$
\end{proposition}
\begin{proof}
Suppose that a point $x \ne 0$ is a periodic point of the smallest
period in the hole $I_{i,N}$. Let that period is equal to $p > 1$
(the case $p=1$ is considered separately). All periodic points have
the form $\frac{l}{2^k-1}$, $k \in \mathbb N$, $l \in \mathbb N$, $0
\le l \le 2^k-1$. The endpoints of the elements of a Markov
partition have the form $\frac{n}{2^s}$, $k \in \mathbb N$, $n \in
\mathbb N$, $0 \le n \le 2^k-1$. Therefore all periodic points
except for zero and one are in the interior of the elements of the
partition. Thus
$$
\lambda \left( \hat{T}^p(I_{i,N}) \cap I_{i,N} \right) > 0,
$$
and, therefore, $\tau (I_{i,N}) \le p$.

To obtain the opposite inequality we need to use the Markov property of the partition.
Suppose that $\lambda \left( \hat{T}^k(I_{i,N}) \cap I_{i,N} \right) > 0$ for some $k \le p$.
Since partition is Markov we have $I_{i,N} \subset \hat{T}^k(I_{i,N})$.
Thus by the standard application of Intermediate Value Theorem to $\hat{T}^k$
we conclude that $I_{i,N}$ contains a periodic point of period $k$.
Therefore, $\tau (I_{i,N}) \ge p$.
This finishes the proof for the case $x \ne 0$.

Now consider the case of a fixed point, $x = 0$ (the case of $x=1$ can be treated similarly).
We only need to consider one hole, $I_{0,N}$.
Clearly,
$$
\lambda \left( \hat{T}(I_{0,N}) \cap I_{0,N} \right) =  \lambda \left( I_{0,N} \right)  > 0,
$$
and, therefore, $\tau (I_{0,N}) = 1$.
\end{proof}

\begin{proposition}
$$
\rho(I_{i,N}) = - \lim_{n \rightarrow \infty} \frac{1}{n} \ln \frac{f^n_{i,N}}{2^n},
$$
if the limit exists.
\end{proposition}
\begin{proof}
Consider a Markov partition $\mathcal I_N$ of a unit interval. There are $f^n_{i,N}$ elements
of this partition that do not enter the hole $I_{i,N}$ in the first $n$ iterations of $T$.
The measure of each element is $\frac{1}{2^{n+N}}$. Thus
$$
1 - \lambda\left( \Omega_n (I_{i,N}) \right) = \frac{f^n_{i,N}}{2^{n+N}}.
$$
Therefore,
$$
\rho(I_{i,N}) = - \lim_{n \rightarrow \infty} \frac{1}{n} \ln \left( 1 - \lambda\left(\Omega_n (I_{i,N}) \right) \right) = - \lim_{n \rightarrow \infty} \frac{1}{n} \ln \frac{f^n_{i,N}}{2^n}.
$$
\end{proof}

We now describe the distribution of periodic points among different holes.
But at first we look at the distribution of the periodic points in the whole interval.
The following Proposition is a well known result(see, e.g. Proposition 1.7.2 in \cite{Katok}).

\begin{proposition} \label{p_455}
The number, $p(k)$, of periodic points of period $k$ (not necessary minimal) of
the doubling map of the unit interval
is equal to $2^k-1$ and the distance between two neighboring
periodic points of the same period is equal to $\frac{1}{2^{k}-1}$.
\end{proposition}
In other words, periodic points of the same period
are distributed uniformly in the unit interval.
Therefore, short intervals have few periodic points of small periods.
In particular the following statements hold.

\begin{corollary}\label{p_32978}
An interval of the size $\frac{1}{2^{N}}$, $N \ge 1$,
contains at most one periodic point of a period $k \le N$.
\end{corollary}

\begin{corollary} \label{l_65435373}
Suppose $x$ is a non-periodic point. Then for any positive integer $n$
there exists $\delta(x,n) > 0$ such that $\delta$-neighborhood of $x$
does \emph{not} contain any periodic points
of periods less or equal $n$.
\end{corollary}
\begin{proof}
Proposition \ref{p_455} claims that doubling map has
finitely many periodic points of periods less or equal $n$.
Hence $\delta(x,n) = \min_{y \in Per_k, k \le n} |x-y|$, where $Per_k$ is the set
of all periodic points of period $k$, is well defined.
Moreover, the interval
$\left( x - \delta(x,n), x + \delta(x,n) \right)$ does not contain any periodic points
of period less or equal to $n$.
Recall that all numbers are considered $\mod 1$ and we identify $0$ and $1$.
\end{proof}

\subsection{Symbolic Dynamics.}
We can view all real numbers between 0 and 1 as binary numbers
represented by one-sided infinite sequences of
zeros and ones. In that case, the result of applying a doubling map
is a number whose binary representation is obtained from the original one by
erasing the first symbol and leaving the rest unchanged.
This allows us to introduce the symbolic dynamics for the map under study.

Let $\Omega(m)$ be a finite alphabet (set of symbols) of size $m$.
A word $w$ is a sequence of
symbols from $\Omega(m)$ of a finite or infinite length, $w = \{w_i\}_{i=1}^{j}, w_i \in \Omega(m)$.
Let $|w|$ denote the length of a word $w$.
The set $\Lambda^{+}_{\Omega(m)}$ consists of all one-sided infinite words,
i.e  $\Lambda^{+}_{\Omega(m)} = \left\{ \{ w_i \}_{i=1}^{\infty} : w_i \in \Omega(m) \right\}$.
For a word $w=\{ w_i \}_{i=1}^{k}$ of a finite length $|w|=k$ and a positive integer $n$ define a
\emph{cylinder set} in $\Lambda^{+}_{\Omega(m)}$,
$$
C_w(n) = \left\{W \in \Lambda^{+}_{\Omega(m)} : W_{n+i-1} = w_i, i = 1 \ldots |w| \right\} \subset \Lambda.
$$
Consider now a Bernoulli measure $\hat{\lambda}$
on the collection of cylinder sets and extend it
to the $\sigma$-algebra generated by this collection (see, for example, \cite{Katok}).
In particular, the measure of the cylinder $C_w(1)$ is then given by
$$
\hat{\lambda} (C_w(1)) = m^{-k}.
$$

The shift map of $\Lambda$ into itself is defined as
$\left( \sigma(w) \right)_i = w_{i+1}$,
i.e. $\sigma$ drops the first symbol and shifts the whole sequence to the left.
The shift map preserves the Bernoulli measure.
Then the triplet $\left( \Lambda^{+}_{\Omega(m)},  \sigma , \hat{\lambda} \right)$
defines a measurable dynamical system.

The doubling map is metrically equivalent to this one-sided shift on
the space of infinite binary ($m=2$) sequences (see, for example
\cite{Lind} for details). A Markov hole (see Section
\ref{s_3254125}) of the size $2^{-N}$ corresponds to a cylinder
defined by a word $w$ of the size $N$ and located at the first
position, $C_w (1)$.
The periodic points for the doubling map correspond to periodic
words in the symbolic space.

\subsection{Escape rate.} \label{s_51223415}

Suppose that Markov hole $I_{i,N}$ corresponds to the cylinder $C_w(1)$, $|w|=N$.
Then the set of points that do not enter the hole during the first $n$ iterations of the
doubling map corresponds to the set of points in $\Lambda^{+}_{\Omega(m)}$
that do not enter the $C_w(1)$ after applying a shift map $n$ times,
i.e. infinite words that do not contain the word $w$
in the first $n+N$ positions. Let $c_w(n+N)$ be the number of such infinite words,
$$
c_w(k) = card \left\{v \in \Lambda^{+}_{\Omega(m)} : \sigma^i (v) \notin C_w(1), i=1 \ldots k - | w | + 1 \right\}.
$$

Since the escape rate for the doubling map into the hole $I_{i,N}$ equals to the escape rate
for the shift map into a corresponding cylinder set $C_w(1)$ we have that
$$
\rho(I_{i,N}) = \rho(C_w(1)) = - \lim_{n \rightarrow \infty} \frac{1}{n} \ln \frac{c_w(n+N)}{2^{n+N}},
$$
if the limit exists. The next lemma shows that this limit exists indeed.

\begin{lemma} \label{l_56453}
The escape rate $\rho(C_w(1))$ is well defined and depends only on $w$.
Moreover,
\begin{equation}\label{eq_289890234}
\rho(C_w(1)) = - \lim_{n \rightarrow \infty} \frac{1}{n} \ln \frac{c_w(n+N)}{2^{n}} = - \ln \frac{\theta_w}{2},
\end{equation}
where $\theta_w < 2$ is a constant depending on $w$.
\end{lemma}
\begin{proof}
It is known \cite{GuiOdl1} that there exists a positive integer $n_0$
such that for all $n \ge n_0$
\begin{equation}\label{e_764892370}
c_1 \theta_w ^n \le c_w (n) \le c_2 \theta_w ^n,
\end{equation}
for some constants $c_1$, $c_2$, and $\theta_w < 2$ that depend only on $w$.
Therefore,
\begin{align*}
\frac{n+N}{n} \ln \frac{\theta_w}{2} + \frac{1}{n} \ln c_1 \le \frac{1}{n} \ln \frac{c_w(n+N)}{2^n}  \\
\le \frac{n+N}{n} \ln \frac{\theta_w}{2} + \frac{1}{n} \ln c_2.
\end{align*}

By letting $n$ go to infinity we obtain that
$$
\lim_{n \rightarrow \infty} \frac{1}{n} \ln \frac{c_w(n+N)}{2^n} = \ln \frac{\theta_w}{2} .
$$
\end{proof}

Our next goal is to determine how does $c_w(k)$ depend on $w$.

\subsection{Combinatorics on words.} \label{s_123456}
As we have seen above in order to compute escape rate we need to
count the number of binary words of a fixed length that do not
contain a certain subword, $c_w(k)$. We use some results from the
theory of combinatorics on words to obtain this number.

In \cite{GuiOdl1} Guibas and Odlyzko studied (introduced, according
to Guibas and Odlyzko, by J. Convay) a function from the set of
finite words to itself called an autocorrelation function,
$corr(w)$. Let $w$ be a binary word of the size $k$. Then the value
$corr(w) = [b_1 \ldots b_k]$ is determined in the following way.
Place a copy of the word $w$ under the original and shift it to the
right by $l$ digits. If the overlapping parts match then $b_{l+1} =
1$. Otherwise we set $b_l = 0$. In particular, we always have
$b_1=1$ (since the word matches itself).

Consider the following example. Suppose that $w=[10100101]$, then we have
$$
\begin{tabular}{|c|ccccccccccccccc|c|}
\cline{1-1}\cline{17-17}
$l$ &  &  &  &  &  &  &  &  &  &  &  &  &  &  &  & $b_{l+1}$ \\ \hline
0 & \textbf{1} & \textbf{0} & \textbf{1} & \textbf{0} & \textbf{0} & \textbf{1} & \textbf{0} & \textbf{1} & \multicolumn{1}{|c}{} &  &  &  &  &  &  & 1 \\ \hline
1 &  & 1 & 0 & 1 & 0 & 0 & 1 & 0 & \multicolumn{1}{|c}{1} &  &  &  &  &  & & 0 \\ \hline
2 &  &  & 1 & 0 & 1 & 0 & 0 & 1 & \multicolumn{1}{|c}{0} & 1 &  &  &  &  & & 0 \\ \hline
3 &  &  &  & 1 & 0 & 1 & 0 & 0 & \multicolumn{1}{|c}{1} & 0 & 1 &  &  &  & & 0 \\ \hline
4 &  &  &  &  & 1 & 0 & 1 & 0 & \multicolumn{1}{|c}{0} & 1 & 0 & 1 &  &  & & 0 \\ \hline
5 &  &  &  &  &  & \textbf{1} & \textbf{0} & \textbf{1} & \multicolumn{1}{|c}{0} & 0 & 1 & 0 & 1 &  &  & 1 \\ \hline
6 &  &  &  &  &  &  & 1 & 0 & \multicolumn{1}{|c}{1} & 0 & 0 & 1 & 0 & 1 & & 0 \\ \hline
7 &  &  &  &  &  &  &  & \textbf{1} & \multicolumn{1}{|c}{0} & 1 & 0 & 0 & 1& 0 & 1 & 1 \\ \hline
\end{tabular}
$$
We can now see that $corr(w) = [10000101]$.

Note that we can also view $corr(w)$ as a binary number and slightly abusing notations we denote both,
a binary word and a binary number,
$$
corr(w) = b_1 \cdot 2^{k-1} + b_2 \cdot 2^{k-2} + \ldots + b_k.
$$
In the example above we have
$$
corr([10100101]) = 1 \times 2^7 + 0 \times 2^6 + 0 \times 2^5 + 0 \times 2^4 + 0 \times 2^3 + 1 \times 2^2 + 0 \times 2^1 + 1 \times 2^0.
$$

Similarly, we can define a correlation polynomial,
\begin{equation} \label{eq_456235662}
f_w(z) = b_1 \cdot z^{k-1} + b_2 \cdot z^{k-2} + \ldots + b_k, \quad z \in \mathbb C,
\end{equation}
so that $f_w(2) = corr(w)$.

Autocorrelation function, among other things, describes periodicities in the word.
Consider a cylinder set $C_w(1)$ generated by the word $w$ and suppose
that $C_w(1)$ contains a periodic point
$v = \{v_i \}_{i=1}^{+\infty} \in C_w(1)$ of period $\ell$, $\ell < |w|$.
Since $v$ is a periodic point we have $v_i = v_{i+\ell}$, $\forall i \in \mathbb N$, that is
$$
v=[\underset{|w|}{\underbrace{v_{1} \ldots v_{\ell} \ldots }}  v_{1} \ldots v_{\ell} \ldots].
$$
Therefore, if $corr (w) = [b_1 \ldots b_k]$ then $b_{\ell +1} = 1$.
(The autocorrelation function was used in \cite{Lind2} for computing
of the dynamical Zeta function of subshifts of finite type.)

In order to compute the number of words of the size $n$ avoiding a
given word $w$ of a length $k$, $c_w(n)$, consider a generating
function for $c_w(n)$,
$$
F_w(z) = \sum_{n=0}^{\infty} c_w(n) z^{-i}.
$$
It was shown in \cite{GuiOdl1} that $F_w(z)$ is a rational function
and the following asymptotic estimates on $c_w(n)$ were obtained.

\begin{lemma}\label{l_3743875}
Suppose that $w$ and $u$ are two words of the same length and
$corr(w) > corr(u)$. Then,
$$
\lim_{n \to \infty} \frac{\ln c_w(n)}{n} > \lim_{n \to \infty}
 \frac{\ln c_u(n)}{n}.
$$
\end{lemma}

The following non-asymptotic result relating the number of sequences that do not contain certain word to the
autocorrelation function of that word was proved in \cite{Eri2}.

\begin{lemma}\label{l_324123}
Suppose that $w$ and $u$ are two words of the same length and $corr(w) > corr(u)$. Then there exists
$\tilde{n}_0$, such that for all $n \ge \tilde{n}_0$ one has that
\begin{equation}\label{eq_4353676} c_w(n) > c_u(n).
\end{equation}
\end{lemma}
Then in \cite{Man1} and \cite{Man2} this result was improved by finding an explicit value of $\tilde{n}_0$ and
expanding the result to the words of different lengths and to the systems with the alphabet of any finite size.
Specifically, consider binary words of equal length, $w$ and $u$. Suppose that $corr (w) = [b_1 \ldots b_N]$ and
$corr (u) = [a_1 \ldots a_N]$. Then,
\begin{equation} \label{eq_82451292}
\tilde{n}_0 =  N + \min \{ i : b_i \ne a_i \} -1.
\end{equation}

\subsection{Main result.} \label{s_51223416}
Lemma \ref{l_324123} leads to the following relationship between the correlation function on one hand and the
escape rate into and survival probability of the cylinder generated by the corresponding word on the other hand.

\begin{lemma} \label{l_56345467}
Suppose that $w$ and $u$ are two words of the same length. Let $C_w(1)$  and $C_u(1)$ be two cylinder sets
generated by these two words. Then
$$
corr(w) > corr(u) \Rightarrow \rho(C_w(1)) < \rho(C_u(1)).
$$
\end{lemma}
\begin{proof}
By Lemma \ref{l_56453}
\begin{align*}
\rho(C_u(1)) = - \lim_{n \to \infty} \frac{1}{n} \ln \frac{c_u(n+N)}{2^{n}}, \\
\rho(C_w(1)) = - \lim_{n \to \infty} \frac{1}{n} \ln \frac{c_w(n+N)}{2^{n}},
\end{align*}

Using the results of Lemma \ref{l_3743875} we complete the proof.
\end{proof}

As before, the number $1 - \lambda \left( \Omega_n (C_w(1))\right)$ is called a survival probability of the set
$C_w(1)$.

\begin{lemma} \label{l_6457565368}
Suppose that $w$ and $u$ are two words of the same length with $corr(w) > corr(u)$, Let $corr (w) = [b_1 \ldots
b_N]$ and $corr (u) = [a_1 \ldots a_N]$. Let $C_w(1)$  and $C_u(1)$ be two cylinder sets generated by these two
words. Then for all $n \ge \min \{ i : b_i \ne a_i \} -1$,
$$
1 - \lambda \left( \Omega_n (C_w(1))\right) > 1 - \lambda \left( \Omega_n (C_u(1))\right).
$$
\end{lemma}
\begin{proof}
In view of \ref{eq_4353676} there exists $\tilde{n}_0>0$ such that $c_w(n+N)
> c_u (n+N)$ for all $n + N \ge \tilde{n}_0$. Therefore,
$$
\frac{c_w(n+N)}{2^{n+N}} >  \frac{c_u(n+N)}{2^{n+N}},\quad n + N \ge \tilde{n}_0.
$$
But,
$$
1 - \lambda \left( \Omega_n (C_w(1))\right) =\frac{c_w(n+N)}{2^{n+N}}, \quad 1 - \lambda \left( \Omega_n
(C_u(1))\right)=\frac{c_u(n+N)}{2^{n+N}}.
$$
It follows from Equation \ref{eq_82451292} that we must have $ n + N \ge N + \min \{ i : b_i \ne a_i \} -1$.
Therefore,
$$
n \ge \min \{ i : b_i \ne a_i \} - 1.
$$
\end{proof}

Finally, we turn to the proof of the main theorems of this section, which state that the escape rate is larger
for the hole that has a larger Poincar\'{e} recurrence time (asymptotic result). Moreover, the survival
probability is smaller for the hole that has a larger Poincar\'{e} recurrence time (non-asymptotic result).

\begin{figure}
\centering
\includegraphics[width=4in]{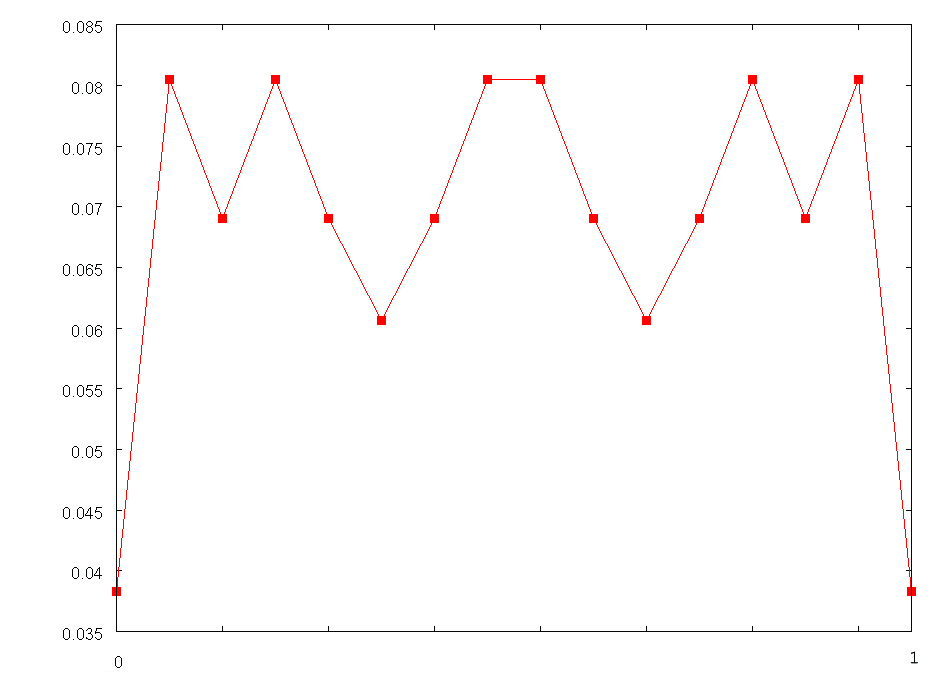}
\caption{Escape rate vs. position of the Markov hole of size $2^{-N}$, $N=4$.} \label{fig_esc_16}
\end{figure}
\begin{theorem} \label{thm_34213}
Let $I_{i,N}$ and $I_{j,N}$ be two Markov holes for the doubling map of a unit interval. Suppose that
$\tau(I_{j,N}) > \tau(I_{i,N})$. Then,
$$
\rho(I_{j,N}) > \rho(I_{i,N}).
$$
Moreover, for all $n \ge  \tau(I_{i,N})$,
$$
1 - \lambda \left( \Omega_n (I_{j,N})\right) < 1 - \lambda \left(\Omega_n (I_{i,N})\right).
$$
\end{theorem}
\begin{proof}
Let $w$, $|w|=N$, be the word that codes the hole $I_{i,N}$ and let $C_{w}(1)$ be the cylinder generated by that
word. Similarly, let $u$, $|u|=N$, be the word that codes the hole $I_{j,N}$ and let $C_{u}(1)$ be the cylinder
generated by that word. Consider the autocorrelation functions of $w$ and $u$, $corr (w) = [b_1 \ldots b_N]$ and
$corr (u) = [a_1 \ldots a_N]$. As we have seen before, $b_1 = a_1 = 1$. Also, $b_{\ell+1} = 1$ if and only if
$C_{w}(1)$ contains a periodic point of period $\ell$, $1 \le \ell < N$. The same is true for $a_{\ell+1}$.
Hence,  the first non-zero element of $[b_1 \ldots b_N]$ after $b_1$ is $b_{\tau(I_{i,N})+1}$. Thus, if
$\tau(I_{j,N}) > \tau(I_{i,N})$ then $corr (u) < corr (w)$. But by Lemma \ref{l_56345467}
$$
\rho(I_{j,N}) = \rho(C_{u(1)}) > \rho(C_{w(1)}) = \rho(I_{i,N}).
$$

In order to compare surviving probabilities we use Lemma \ref{l_6457565368}. In addition,  the argument in the
preceding paragraph shows that $\min \{ i : b_i \ne a_i \} = \tau(I_{i,N}) + 1$. Since $\lambda \left( \Omega_n
(I_{j,N})\right) = \lambda \left( \Omega_n (C_u(1))\right)$ and $\lambda \left( \Omega_n (I_{i,N})\right) =
\lambda \left( \Omega_n (C_w(1))\right)$, we obtain that
$$
1 - \lambda \left( \Omega_n (I_{j,N})\right) < 1 - \lambda \left(\Omega_n (I_{i,N})\right),
$$
for all $n \ge  \tau(I_{i,N})$.
\end{proof}

\subsection{Local escape rate.}\label{sec_4534345234}

Our next task is to investigate what happens to the escape rate as
the size of the hole is decreasing. Let $I_{N}(x)$ be an element of
a partition $\mathcal I_N$ (see Section \ref{s_3254125a}) that
contains a point $x$.

It is known \cite{Afr2} that for a large class of symbolic systems,
which includes the expanding maps of the interval considered here,
for almost every point Poincar\'{e} recurrence time grows linearly
with $N$ as a size of the hole exponentially decreases, $\lim_{N
\rightarrow \infty} \frac{\tau{(I_{N}(x)})}{N} = 1$.

We want to obtain a similar result for the escape rate.
The following theorem answers this question.
In particular, it says that
the escape rate decreases linearly with respect to the size of the hole.

\begin{figure}
\centering
\includegraphics[width=4in]{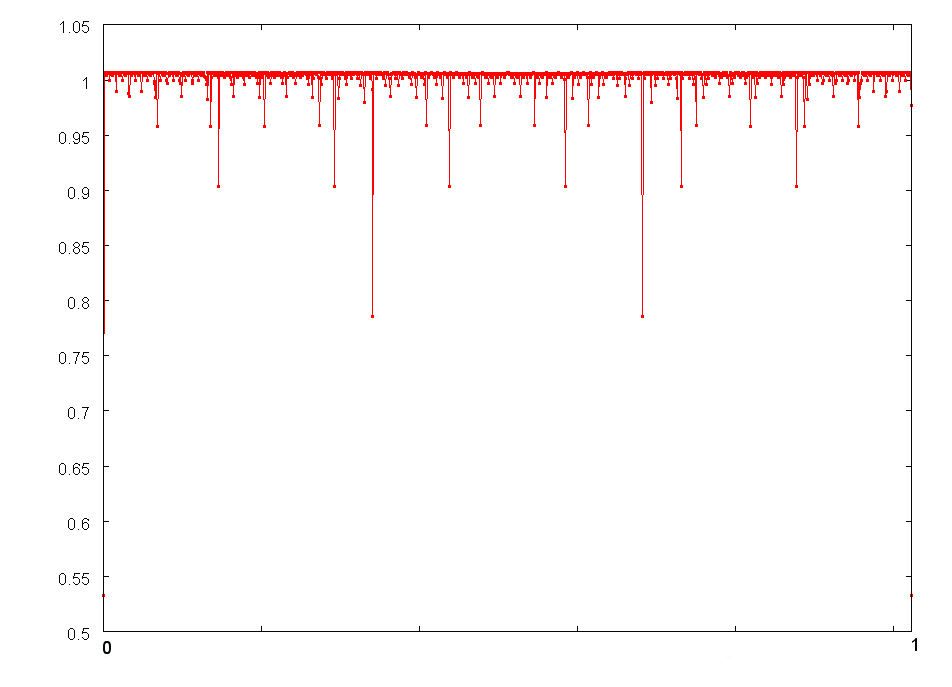}
\caption{Plot of $\frac{\rho(I_{N}(x))}{\lambda(I_{N}(x))}$ vs $x$, $N=10$.}
\label{loc_1024}
\end{figure}
\begin{theorem}\label{thm_55252512}
Consider the doubling map of the unit interval. Suppose that $x \in
[0,1]$ . Then the following statements are true:
\begin{enumerate}
    \item if $x$ is a periodic point of period $m$ then
$$
\lim_{N \to \infty} \frac{\rho(I_{N}(x))}{\lambda(I_{N}(x))} = 1 -
\frac{1}{2^m};
$$
    \item if $x$ is a non-periodic point then
$$
\lim_{N \to \infty} \frac{\rho(I_{N}(x))}{\lambda(I_{N}(x))} = 1.
$$
\end{enumerate}
\end{theorem}

\begin{proof}
Using equality \ref{eq_289890234} and the fact that $\lambda(I_{N}(x)) =2^{-N}$, we obtain,
$$
\lim_{N \to \infty} \frac{\rho(I_{N}(x))}{\lambda(I_{N}(x))}
= - \lim_{N \to \infty} 2^N ln \frac{\theta_w}{2},
$$
where $\theta_w$ depends on $N$.

Recall that for any word $w$ of a length $N$ we can define a correlation polynomial
\ref{eq_456235662} as
$$
f_w(z) = \sum_{i=1}^N b_i z^{N-i},
$$
where $corr (w) =[b_1, \ldots, b_N]$, $b_i$ depends on $N$, $b_i = b_i(N)$ ,
and $z \in \mathbb C$. Clearly, $b_1 = 1$, thus $deg f_w (z) = N-1$.

It was shown in \cite{GuiOdl3} that asymptotically the constant $\theta_w$ satisfies
$$
\ln \theta_w = \ln 2 - \frac{1}{2 f_w(2)} + \mathcal O (2^{-2N}).
$$
Then,
$$
\ln \theta_w = \ln 2 - \frac{1}{2^N + \sum_{i=\tau (I_{N}(x))}^N
b_i(N) 2^{N-i}} + \mathcal O (2^{-2N}).
$$
Suppose now that $x$ is non-periodic point. Then one has that,
$$
\lim_{N \to \infty} 2^{-N} \sum_{i=\tau (I_{N}(x))}^N b_i(N) 2^{N-i}
= 0.
$$
Indeed,
$$
2^{-N} \sum_{i=\tau (I_{N}(x))}^N b_i 2^{N-i} = \sum_{i=\tau
(I_{N}(x))+1}^N b_i(N) 2^{-i} \le 1.
$$
But, $\lim_{N \to \infty}b_i(N) = 0$ for all $i \ge \tau
(I_{N}(x))$.

Therefore,
\begin{align*}
\lim_{N \to \infty} \frac{\rho(I_{N}(x))}{\lambda(I_{N}(x))} =\lim_{N \to \infty} \frac{2^N}{2^N +   \sum_{i=\tau (I_{N}(x))}^N b_i(N) 2^{N-i}} \\
= \lim_{N \to \infty} \frac{1}{1 +  2^{-N}\sum_{i=\tau (I_{N}(x))}^N
b_i(N) 2^{N-i}} =1.
\end{align*}

Suppose now that $x$ is a periodic point of (minimum) period $m$.
Then,
\begin{align*}
\lim_{N \to \infty} 2^{-N} \sum_{i=\tau (I_{N}(x))}^N b_i(N) 2^{N-i} \\
= \sum_{k=1}^{\infty} 2^{-km} = \frac{1}{1-2^{-m}} - 1.
\end{align*}
Therefore,
\begin{align*}
\lim_{N \to \infty} \frac{\rho(I_{N}(x))}{\lambda(I_{N}(x))} =\lim_{N \to \infty} \frac{2^N}{2^N +   \sum_{i=\tau (I_{N}(x))}^N b_i(N) 2^{N-i}} \\
= \frac{1}{1 +  \lim_{N \to \infty} \left( 2^{-N}\sum_{i=\tau (I_{N}(x))}^N b_i(N) 2^{N-i} \right)} \\
=1 - \frac{1}{2^m}.
\end{align*}

This completes the proof.

\end{proof}

\begin{remark}
All non-periodic points form a set of full measure. Thus for almost
every point in the interval $[0,1]$ the escape rate into the hole
"centered" at a given point asymptotically, as the size of the hole
goes to zero, equals to the measure of the hole.
\end{remark}

Suppose that an interval $A$ does not contain points $x = s2^{-k}$
for all $k \le n$ for some $n>1$, $s, k,n \in \mathbb N$, i.e. $A$
does not contain an end point of an element of any Markov partition
$\mathcal I_N$, $N \le n$. Then one can find two elements of the
Markov partitions $I_{N_1}(x)$ and $I_{N_2}(x)$ so that
$$
I_{N_1}(x) \subseteq A \subseteq I_{N_2}(x).
$$
Thus, the following result holds for the arbitrary decreasing nested sequence of intervals.

\begin{figure}
\centering
\includegraphics[width=4in]{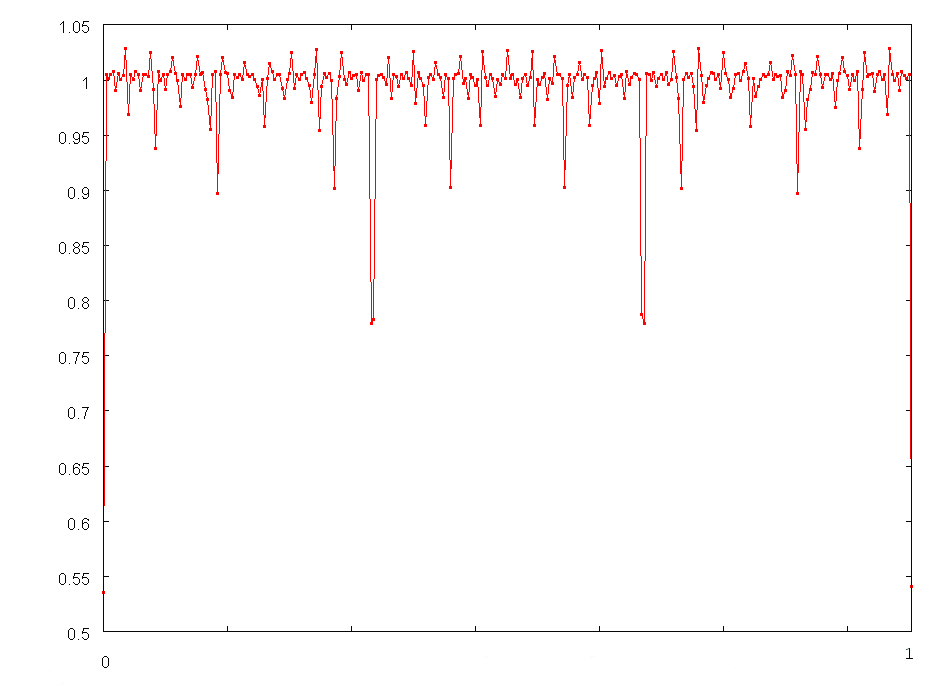}
\caption{Plot of $\frac{\rho(A(x))}{\lambda(A(x))}$ vs. $x$, $\lambda(A(x))=\frac{1}{327}$.}
\label{loc_327}
\end{figure}
\begin{corollary}\label{c_653457567}
Consider the doubling map of a unit interval. Let $x \in [0,1]$ and suppose $\{A_n (x) \}_{n=1} ^{\infty}$ is a
sequence of nested decreasing intervals with $x = \cap _{n=1} ^{\infty} A_n$ for all $n$. The following
statements are true:
\begin{enumerate}
    \item if $x$ is a periodic point of period $m$, then
$$
\lim_{n \to \infty} \frac{\rho(A_{n}(x))}{\lambda(A_{n}(x))} = 1 - \frac{1}{2^m};
$$
    \item if $x$ is a non-periodic point and $x \neq s2^{-k}, s,k \in \mathbb Z^{+}$, then
$$
\lim_{n \to \infty} \frac{\rho(A_{n}(x))}{\lambda(A_{n}(x))} = 1.
$$
\end{enumerate}
\end{corollary}

\begin{remark}
Observe that in Fig. \ref{loc_327} there are many local maxima of $\frac{\rho(A(x))}{\lambda(A(x))}$ which are
essentially larger than one. These happen at the places where the intervals $A(x)$ have maximum Poincar\'{e}
return time among all such intervals $A(x)$ with length of $A(x)$ being fixed. Indeed, consider a set of holes
that are Markov, say $\mathcal I_N$. The ratio $\frac{\rho(I_{i,N})}{\lambda(I_{i,N})}$ attains its maximum (see
Corollary \ref{cor_23234} below) in the holes that have the largest Poincar\'{e} return time (in this case $N$).
Now, suppose that $\lambda(I_{i,N+1}) \le \lambda(A(x)) \le \lambda(I_{i,N})$, $A(x) \subset I_{i,N}$, and
$\tau(A(x)) = \tau(I_{i,N})$. Then $\rho(A(x)) \approx \rho(I_{i,N})$, but
$$
\frac{\rho(A(x))}{\lambda(A(x))} \ge \frac{\rho(I_{i,N})}{\lambda(I_{i,N})}.
$$
Thus, by shrinking a Markov hole while keeping the same Poincar\'{e} return time we can increase the ratio of the
escape rate to the size of the hole.
\end{remark}

\subsection{More results for one hole.}\label{sec_4534626}

In the case of the doubling map with Markov holes we know precisely where to make a hole
to achieve maximum (or minimum) escape rate.

\begin{corollary}\label{cor_23234}
$$
\min_{1 \le i \le 2^N} \rho (I_{i,N}) = \rho(I_{1,N}), \quad \max_{1 \le i \le 2^N} \rho (I_{i,N}) = \rho(I_{2,N}),
$$
although the holes that give these extremes are not unique.
\end{corollary}
\begin{proof}
Clearly, by Proposition \ref{p_32978}, one has that $1 \le \rho(I_{i,N}) \le N$.
Then, it is easy to check that $\rho(I_{1,N})=1$ and $\rho(I_{2,N}) = N$.
Thus the result follows from the Main theorem.
\end{proof}

Note that we obtained minimum escape rate in one more interval that contain a fixed point,
namely $I_{2^N,N}$. The maximum is obtained in the intervals that have a minimum period equal to $N$,
as Figure \ref{fig_esc_16} illustrates this.

\begin{figure}
\centering
\includegraphics[width=4in]{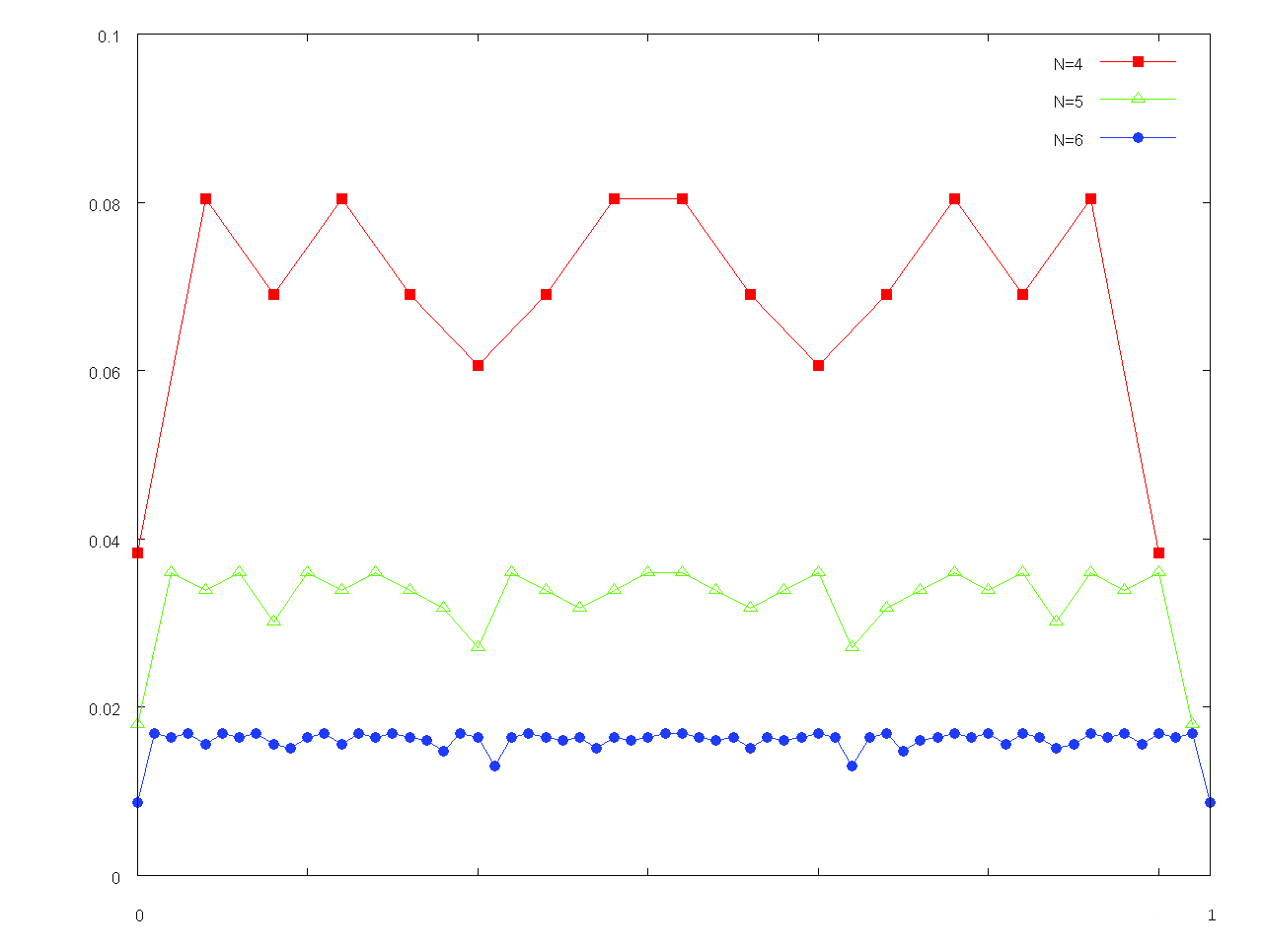}
\caption{Escape rate vs. position of the Markov hole of size $2^{-N}$.}
\label{mono-16-64}
\end{figure}
Next theorem states that the escape rate decreases monotonically as we decrease the size of the hole.
\begin{theorem}
$$
\max_{1 \le i \le 2^{N+1}} \rho (I_{i,N+1}) = \min_{1 \le i \le 2^N} \rho (I_{i,N}).
$$
\end{theorem}
\begin{proof}
It follows from Corollary \ref{cor_23234} that
$$
\max_{1 \le i \le 2^{N+1}} \rho (I_{i,N+1}) = \rho(I_{2,N+1}), \quad \min_{1 \le i \le 2^N} \rho (I_{i,N}) = \rho(I_{1,N}).
$$
Let $w_1$ and $w_2$ be two binary words that define holes $I_{2,N+1}$ and $I_{1,N}$, respectively.
then $|w_1|=N+1$ and $|w_2|=N$.

We now proceed in the following fashion.
First, we compute the generating functions
$F_{w_1}(z) = \sum_{j=0}^{\infty} c_{w_1}(j) z^{-j}$ and
$F_{w_2}(z) = \sum_{j=0}^{\infty} c_{w_2}(j) z^{-j}$ defined in Section \ref{s_123456}.
Next, we show that $c_{w_1}(n) \sim c_{w_2}(n)$ for $n \gg 1$.
Finally, using Lemma \ref{l_56453} we conclude that $\rho (I_{2,N+1}) =  \rho (I_{1,N})$.

The explicit analytic expression for the generating function was
found in \cite{GuiOdl1},
$$
F_{w}(z) = \frac{z \cdot f_w(z)}{1+(z-2) \cdot f_w(z)},
$$
where $f_w(z)$, as before, is a correlation polynomial of $w$.
It is easy to check that
$$
w_1 = \underset{N}{\underbrace{0 \ldots 0 }}1, \quad w_2 = \underset{N}{\underbrace{0 \ldots 0 }}
$$
and thus,
$$
corr(w_1) =[1 \underset{N}{\underbrace{0 \ldots 0 }}], \quad corr(w_2) =[\underset{N}{\underbrace{1 \ldots 1 }}].
$$
Therefore,
$$
f(w_1) = z^{N}, \quad f(w_2) = \sum_{j=0}^{N-1} z^j.
$$
After some tedious but straightforward algebra we arrive at
\begin{align*}
F_{w_1}(z) = \frac{1}{ 1 - (2 z^{-1}-z^{-(N+1)})}, \\
F_{w_2}(z) = \frac{1}{ 1 - (2 z^{-1}-z^{-(N+1)})} - \frac{z^{-N}}{ 1 - (2 z^{-1}-z^{-(N+1)})}.
\end{align*}

Let $t=z^{-1}$ and expand the above equations into power series. Then we get
\begin{align*}
F_{w_1}(z) = \sum_{j=0}^{\infty} \left( 2t - t ^{N+1}\right)^j,\\
F_{w_2}(z) = \sum_{j=0}^{\infty} \left( 2t - t ^{N+1}\right)^j - t^N \cdot \sum_{j=0}^{\infty} \left( 2t - t ^{N+1}\right)^j.
\end{align*}
Therefore, $\sum_{j=0}^{\infty} \left( 2t - t ^{N+1}\right)^j = \sum_{j=0}^{\infty}c_{w_1}(j) t^j$.
Then,
\begin{align*}
F_{w_1}(z) = \sum_{j=0}^{\infty}c_{w_1}(j) t^j,\\
F_{w_2}(z) = \sum_{j=0}^{N-1}c_{w_1}(j) t^j + \sum_{j=N}^{\infty}\left(c_{w_1}(j) - c_{w_1}(j-N) \right)t^j.
\end{align*}
Thus, for $j > N$ we obtain the following relationship:
$$
c_{w_2}(j) = c_{w_1}(j) - c_{w_1}(j-N).
$$

The equation \ref{e_764892370} implies that for $j \gg 1$,
$$
C_1 \theta_{w_1} ^j \le c_{w_1} (j) \le C_2 \theta_{w_1} ^j,
$$
for some constant $C_1$ and $C_2$.
Thus,
$$
\theta_{w_1}^j \left(C_1 - C_2 \theta_{w_1}^{-N} \right) \le c_{w_1}(j) - c_{w_1}(j-N) \le \theta_{w_1}^j \left(C_2 - C_1 \theta_{w_1}^{-N} \right)
$$
Hence, for $j \gg 1$ we have that $c_{w_2}(j)$ asymptotically behaves as $\theta_{w_1}^j$,
$c_{w_2}(j) \sim \theta_{w_1}^j$.
This finishes the proof.
\end{proof}

The next theorem deals with arbitrary (not necessarily the elements
of Markov partition) but sufficiently small holes.

\begin{theorem}
Suppose $x_1$ and $x_2$ are two periodic points with periods $m_1$ and $m_2$, respectively,
$m_1 < m_2$.
Then there exists a positive number $\varepsilon_0 = \varepsilon_0(m_1,m_2)$
such that for all subintervals $A_1$ and $A_2$ of the unit interval with
$$
x_i \in int(A_i), \quad i=1,2, \quad  \lambda(A_1) = \lambda(A_2) < \varepsilon_0,
$$
one has
$$
\rho(A_1) < \rho(A_2).
$$
\end{theorem}
\begin{proof}
The result follows directly from Corollary \ref{c_653457567}.
\end{proof}

\subsection{Contributions to the escape rate: the size of the hole vs. dynamics.}
The doubling map shows that the escape rate is not determined by the size of the hole alone.
The following examples demonstrate the possibility
of a larger escape rate into the smaller hole.

Consider two sets $A = [0,\frac{1}{4})$ and $B=[\frac{1}{4},\frac{1}{2})$.
Then $\lambda(A) = \lambda(B) = \frac{1}{4}$.
It easy to check that $\tau(A)=1$ and $\tau(B)=2$.
Thus it follows from Theorem \ref{thm_34213} that $\rho(A) < \rho(B)$.

On the other hand by Proposition \ref{prop_23423} it follows that
$\rho(\hat{T}^{-1}A \cup A) = \rho(A)$.
Thus, for two holes,
$$
\hat{T}^{-1}A \cup A = \left[0,\frac{1}{4}\right) \cup \left[\frac{1}{2}, \frac{5}{8}\right) \text{ and } B =\left[\frac{1}{4},\frac{1}{2}\right),
$$
one has
$$
\rho \left(\left[0,\frac{1}{4}\right) \cup \left[\frac{1}{2}, \frac{5}{8}\right)\right) < \rho\left(\left[\frac{1}{4},\frac{1}{2}\right)\right), \quad \lambda \left(\left[0,\frac{1}{4}\right) \cup \left[\frac{1}{2}, \frac{5}{8}\right)\right) > \lambda\left(\left[\frac{1}{4},\frac{1}{2}\right)\right).
$$

The next example shows that even if we have two holes which are connected sets
it is still possible to have a faster escape through a smaller one.
Consider two holes $A\cup B$ and $C$, where
$A=[0,\frac{1}{4})$, $B=[\frac{1}{4},\frac{5}{16})$, and $C=[\frac{1}{2},\frac{3}{4})$.
Note that $B \subset \hat{T}^{-2}A$.
Clearly, $\lambda(A\cup B)=\frac{5}{16} > \lambda(C)=\frac{4}{16}$.
It is easy to check that $\tau(A)=1$ and $\tau(C)=3$,
thus $\rho(A) < \rho(C)$. On the other hand, by Proposition \ref{prop_23423}
it follows that $\rho(A) = \rho(A \cup B)$.
Therefore we have that $\rho(A \cup B) < \rho(C)$.
Hence,
$$
\rho \left( \left[0,\frac{5}{16}\right) \right) < \rho \left( \left[\frac{1}{2},\frac{3}{4}\right) \right), \quad \lambda \left( \left[0,\frac{5}{16}\right) \right) > \lambda \left(  \left[\frac{1}{2},\frac{3}{4}\right) \right).
$$

In general, the following result shows that
there are holes of the arbitrarily large size with arbitrarily small escape rate.

\begin{theorem}
For any $\varepsilon \in (0,1)$ and any $r>0$ there exists a measurable set $A \subset [0,1]$
such that
$$
\lambda(A) > 1- \varepsilon, \quad \rho(A) < r.
$$
\end{theorem}
\begin{proof}
We can always pick $N$ such that $\rho(I_{i,N}) < r$ for some $i$.
One has that
$$
\lambda \left( \bigcup _{j=0}^{\infty} \hat{T}^{-j} I_{i,N} \right) = 1,
$$
but
$$
\lambda \left( \bigcup _{j=0}^{n} \hat{T}^{-j} I_{i,N} \right) < 1
$$
for any finite $n$.
Thus for any $\varepsilon > 0$ we can find $n_0$ such that
$$
1 > \lambda \left( \bigcup_{j=0}^{n_0} \hat{T}^{-j} I_{i,N} \right) > 1 - \varepsilon.
$$
By Proposition \ref{prop_23423} one has that $\rho(I_{i,N}) =
\rho(\cup_{j=0}^{n_0} \hat{T}^{-j} I_{i,N})$. Now, set
$$
A = \bigcup_{j=0}^{n_0} \hat{T}^{-j} I_{i,N}.
$$
\end{proof}

\section{Some generalizations.}\label{sec:more}

\subsection{Linear expanding map.}
As was mentioned at the beginning of this section the same results
hold $x \mapsto \kappa x \mod 1$, $\kappa \in \mathbb N$ and $\kappa
> 1$.

\subsection{Tent map.}
The tent map is a map $\hat{T}$ of a unit interval to itself given by
$$
\hat{T}(x) = \left\{
\begin{array}{ll}
2x, & 0\leq x\leq \frac{1}{2} \\
2-2x, & \frac{1}{2}\leq x\leq 1
\end{array}
\right. .
$$


This map preserves a Lebesgue measure on $[0,1]$.
We can use the following correspondence between the tent map and the symbolic dynamics:
$s_n = 0$ if $\hat{T}^n x < 0.5$ and $s_n=1$ otherwise.
It can be easily shown that mapping $x \mapsto \{s_n \}$
is a metric conjugacy onto the left shift symbolic space.
Then one can repeat all the arguments that we used for the doubling map to obtain similar result.
Namely, for $I_{i,N}$ defined as before,
$$
I_{i,N} = \left[ \frac{i-1}{2^N}, \frac{i}{2^N} \right], \quad i=1 \ldots 2^N,
$$
the following statement holds.

\begin{theorem}
Consider a tent map.
$\forall N \in \mathbb N$ and $\forall i,j = 1,.., 2^N$ we have
that if $\tau(I_{j,N}) > \tau(I_{i,N})$ then
$\rho(I_{j,N}) > \rho(I_{i,N})$.
\end{theorem}

\subsection{Logistic map.}
Consider a logistic map $\hat{T}: [0,1] \rightarrow [0,1]$ given by
$$
\hat{T} (x) = 4x(1-x), \quad x \in [0,1].
$$
It preserves a measure, $\lambda$,  (absolutely continuous with respect to Lebesgue measure)
with the density $\frac{1}{\pi \sqrt{x(1-x)}}$.


A logistic map and the tent map are metrically conjugate:
it is easy to check that conjugacy is given by the
transformation $y = \sin^2 \frac{\pi x}{2}$.

Thus, by Lemma  \ref{l_32521},
for $I_{i,N}$ (preimages of a finite Markov partition for $\hat{T}$) defined as
$$
I_{i,N} = \left[ \frac{2}{\pi} \arcsin \sqrt{\frac{i-1}{2^N}}, \frac{2}{\pi} \arcsin \sqrt{\frac{i}{2^N}} \right], \quad i=1 \ldots 2^N,
$$
the following theorem holds.
\begin{theorem}
Consider a logistic map. Suppose that $\tau(I_{j,N}) >
\tau(I_{i,N})$. Then, $\rho(I_{j,N}) > \rho(I_{i,N})$.
\end{theorem}

\subsection{Baker's map.}
A Baker's map is an example of two dimensional invertible hyperbolic
chaotic map. This map, $\hat{T}: [0,1] \times [0,1] \to [0,1] \times
[0,1]$, is defined in the following way:
$$
\hat{T} (x,y) = \left( 2x \mod 1,  \frac{1}{2}(y+\lfloor 2x \rfloor) \mod 1 \right),
$$
where $\lfloor x \rfloor$ is an integer part of $x$.
It preserves the Lebesgue measure and is mixing (and, therefore, ergodic).

It was observed numerically in \cite{Tel} that escape rate in this
system depends on the position of the hole. The Markov holes for
this map are the rectangles given by
$$
I_{i,j,N,M} = \left[ \frac{i-1}{2^N}, \frac{i}{2^N} \right] \times  \left[ \frac{j-1}{2^M}, \frac{j}{2^M} \right], \quad i=1 \ldots 2^N, \quad j=1 \ldots 2^M.
$$

It is well known that Baker's map is conjugate to the full binary shift.
Thus the following statement holds.

\begin{theorem}
Consider the Baker's map $\hat{T}$.
Then $\forall N, M \in \mathbb N$, $\forall i_1,i_2 = 1,.., 2^N$, and $\forall j_1,j_2 = 1,.., 2^M$
we have
that if $\tau(I_{i_1,j_1,N,M}) > \tau(I_{i_2,j_2,N,M})$ then
$\rho(I_{i_1,j_1,N,M}) > \rho(I_{i_2,j_2,N,M})$.
\end{theorem}

\section{Conclusions.}\label{sec:end}

Apparently there are still other natural and interesting unexplored
questions on the dynamics of open dynamical systems. In this paper
we dealt with one such problem, the dependence of the escape rate on
the position of a hole. We demonstrated that dynamical
characteristics can play as important role for the escape rate as
the size (measure) of the hole does.

Here we just scratch the surface of this new area in the theory of
open dynamical systems dealing with the uniformly hyperbolic
systems. In general, the effect of distortion (variability of values
of the derivative or Jacobian of the dynamical system) can also
essentially contribute to the escape through a hole of a given size.
This problem will be addressed in another paper.

\section*{Acknowledgements.}\label{sec:thanks}
Authors are grateful to V. Afraimoivich, M. Demers,  M. Pollicott, and T. T\'{e}l for numerous helpful
discussions.

\end{document}